\documentclass[11pt]{amsart}

\usepackage{amsmath}
\usepackage{amsthm}
\usepackage{amsfonts}
\usepackage{amssymb}
\usepackage[english]{babel}

\newcommand{\de}{\mathrm{d}}

\newcommand{\Sym}{\mathrm{Sym}}

\newcommand{\Def}{\mathrm{Def}}
\newcommand{\Tr}{\mathrm{Tr}\,}

\newtheorem{lemma}{Lemma}

\title[Linearly foliated Calabi--Yau $n$-folds]%
      {Linearly foliated Calabi--Yau $n$-folds\\ I. First-order deformations} 
\author{Antonio Ricco} 

\begin{document}
\begin{abstract}   
We consider classes of noncompact $n$-folds with trivial canonical bundle, 
that are linear foliations on nonsingular projective varieties, 
in general without a projection to the base.
We obtain them as first-order deformations of total spaces of
vector bundles on those varieties.  
\end{abstract}
\address{Center of Mathematical Sciences \\ Zhejiang University \\ Hangzhou}

\maketitle
\thispagestyle{empty}

\tableofcontents

\section*{Introduction}

A very important question in geometry, and also a hard one, is related to the classification
of algebraic varieties with trivial canonical bundle, known as Calabi--Yau $n$-folds,
in particular in dimension three.
In fact, Calabi--Yau $n$-folds can be considered as the complex analogue
of oriented real $n$-manifolds, \cite{Thomas}.
Moreover, the lack of understanding of the Calabi--Yau case is regarded as one of the 
principal gaps in the classification of algebraic threefolds, 
see for example \cite{Gross} and the references therein.

On the other hand, a much simpler question
concerns the local aspect of a Calabi--Yau $n$-fold along an embedded subvariety, 
that is, the description of the tubular neighborhood of a subvariety in a Calabi--Yau $n$-fold. 
In this note we will focus on the case in which the subvariety is 
nonsingular, considering projective Calabi--Yau $n$-folds.

Let $X$ be a nonsingular projective variety with trivial canonical bundle
and let $M$ be a nonsingular subvariety of $X$.
The normal cone to $M$ in $X$ is isomorphic to a vector bundle on $M$, 
with rank equal to the codimension of $M$ in $X$,
and called the normal bundle of $M$ in $X$.
We will study deformations to the normal cone at the first order, 
for details see \cite{Fulton},
that is, first-order deformations of 
total spaces of vector bundles on nonsingular projective varieties.
Moreover we will impose on them the vanishing of the canonical bundle.

This note is organized as follows.
Let $M$ be a nonsingular $m$-dimensional projective variety, and $V \to M$
a rank-$r$ vector bundle.
Wanting to describe the group $H^1(V,T_V)$, see \cite{Artin}, 
we will first notice (Section \ref{sec:exactsequence})
that we have a non-canonical isomorphism of the form 
\begin{eqnarray}
  H^1 (V, T_V) &\simeq& H_{\mathrm{vert}} \oplus H_{\mathrm{horiz}} \nonumber
\end{eqnarray}
with
\begin{eqnarray}
    H_{\mathrm{vert}} &:=&  H^1(M, V \otimes \Sym V^*)/\mathrm{im} (\delta_1 )\nonumber \\ 
    H_{\mathrm{horiz}} &:=&  \mathrm{ker} \left(\delta_2: H^1(M, T_M \otimes \Sym V^*) 
                                        \to H^2 (M, V \otimes \Sym V^*) \right) \nonumber
\end{eqnarray}
where $\Sym V^*$ is the symmetric algebra of the dual bundle of $V$.
In fact, this splitting is given by a long exact sequence in cohomology, induced 
by the Atiyah sequence
\begin{eqnarray}
  0 \to \pi^* V \to T_V \to \pi^* T_M \to 0 , \nonumber
\end{eqnarray}
and $\delta_q : H^q(M, T_M \otimes \Sym V^*) \to H^{q+1} (M, V \otimes \Sym V^*)$
are the usual connecting morphisms.
In such a way, we can reduce the deformation problem
for the total space of $V$ to the study of groups,
$H_{\mathrm{vert}}$ and $H_{\mathrm{horiz}}$, that are
defined in terms of cohomology groups on the base $M$.  Those groups
and the maps among them can be quite explicitely computed in \v{C}ech
cohomology (Section \ref{sec:cech}).

For us a Calabi--Yau variety is a variety $X$ with trivial
canonical bundle, i.e. such that $K_X \simeq \mathcal{O}_X$.
Thus, the total bundle of $V$ is a Calabi--Yau variety when 
$\det V \simeq K_M$. 
Given the previous splitting of $H^1(V, T_V)$ into groups
defined on the base $M$, we would like to understand which deformations 
preserve the Calabi--Yau structure of the total space of $V$.
We will observe that there exists a map (Section \ref{sec:CYcondition})
\begin{eqnarray}
  \partial^{\#} : H^1(V, T_V) \to H^1(V, \mathcal{O}_V) \nonumber
\end{eqnarray}
from the deformations of $V$ to the deformations of its
canonical bundle, whose kernel, in the case in which $V$ is Calabi--Yau,
gives the deformations with trivial canonical bundle.
This map is induced in cohomology by the differential map
\begin{eqnarray}
  \partial: T_V \to \mathcal{O}_V . \nonumber
\end{eqnarray}

Finally, in the last section, we will consider the case of a nonsingular curve 
in a nonsingular Calabi--Yau $n$-fold.

\section{Vector bundles and nonlinear deformations}

\subsection{An exact sequence in cohomology}\label{sec:exactsequence}

Let $\pi: V \to M$ be a rank-$r$ vector bundle over a $m$-dimensional 
nonsingular projective variety $M$. 
We will use the same symbol for the vector bundle, for the 
locally-free sheaf of its sections, and for its total space.

We would like to describe the group $H^1(V, T_V)$, of first order deformations
of the total space of $V$, in terms of groups defined on
the base $M$. To this aim, let us first notice the following.

\begin{lemma}\label{lem:isomorph}
  Let $E \to M$ a vector bundle on $M$ and $\pi^* E \to V$ its pullback bundle
  over the total space of $V$. Then
  \begin{eqnarray}
     H^i (V, \pi^*E) \simeq H^i (M, E\otimes\Sym V^*) , \nonumber
  \end{eqnarray}
  where $V^*$ denotes the dual bundle of $V$ and
  $\Sym V^* := \bigoplus_{d \geq 0} \Sym^d V^*$ is the symmetric 
  $\mathcal{O}_M$-algebra of $V^*$.
\end{lemma}
\begin{proof}
Since $\pi: V \to M$ is an affine morphism,
for any quasi-coherent sheaf $\mathcal{E}$ on $V$ we have 
$H^i ( V, \mathcal{E} ) \simeq H^i ( M, \pi_* \mathcal{E} )$.
Thus, in particular, 
\begin{eqnarray}
   H^i(V, \pi^*E) \simeq H^i(M, \pi_*\pi^* E). \nonumber
\end{eqnarray}
Since $E$ is a locally-free sheaf of finite rank,
we now use the projection formula, that gives $\pi_*\pi^* E = E \otimes \pi_*\mathcal{O}_V$,
and, together with $\pi_*\mathcal{O}_V \simeq \Sym V^{*}$, 
the result.
\end{proof}

Let us also notice that we have the exact sequence of vector bundles  
(\emph{Atiyah sequence})
\begin{eqnarray}\label{eq:atiyahsequence}
  0 \to \pi^* V \to T_V \to \pi^* T_M \to 0.
\end{eqnarray}

\begin{lemma}\label{exactseq}
The sequence (\ref{eq:atiyahsequence}) induces the long exact sequence
\begin{eqnarray}
0 &\to&  H^0(M, V\otimes \Sym V^*) \to H^0(V,T_V) 
       \to  H^0(M, T_M \otimes \Sym V^* ) \to  \nonumber \\
&\to& H^1(M, V\otimes \Sym V^*) \to  \cdots \nonumber
\end{eqnarray}
\end{lemma}

\begin{proof}
Using the previous lemma and the usual long exact sequence
induced in cohomology by (\ref{eq:atiyahsequence}).
\end{proof}

Finally, let us notice 
that the image of the map
\begin{eqnarray}
  H^1(M, V\otimes \Sym V^*) \to H^1(V, T_V) \nonumber 
\end{eqnarray}
gives the deformations for which there exists a projection 
to the base $M$.

\subsection{The sequence in \v{C}ech cohomology}\label{sec:cech}

Let us begin with some notational remark.
In the following we will consider locally-free sheaves $E$ of sections of vector bundles
which have been trivialized on an open covering $\mathcal{U} = \{ U_\alpha \}$, 
so each section over $U_\alpha$ will be identified with vectors of functions $f_\alpha$ 
on $U_\alpha$. 
As usual, changing trivialization we will have to multiply by the appropriate transition function.
As a notational choice, the higher cochains $f_{\alpha_0\alpha_1\dots\alpha_q}$, defined on $U_{\alpha_0\alpha_1\dots \alpha_q}$, 
will be given by vectors of functions using the trivialization on $U_{\alpha_q}$.
In this notation, if $E_{\alpha\beta}$ are the transition functions of $E$ on $U_{\alpha\beta}$,
the \v{C}ech differential 
\begin{eqnarray}
\delta : \check{C}^{\bullet} (\mathcal{U}, E) \to \check{C}^{\bullet} (\mathcal{U}, E) \nonumber
\end{eqnarray}
will act on a $q$-cochain $\{ f_{\alpha_0 \cdots \alpha_q} \}$ as
\begin{eqnarray}
(\delta f)_{\alpha_0 \cdots \alpha_{q+1}} = \left.
              \sum_{j=0}^{q} (-)^j f_{\alpha_0 \cdots \hat{\alpha}_j \cdots \alpha_{q+1}} +
              (-)^{q+1} E_{\alpha_{q+1} \alpha_q} f_{\alpha_0 \cdots \alpha_{q}} 
                       \right|_{U_{\alpha_0 \cdots \alpha_{q+1}}} \ . \nonumber
\end{eqnarray}
When using a different trivialization, we will explicitely indicate it with a superscript, i.e.  
$$f_{\alpha_0\dots\alpha_k\dots\alpha_q}^{(\alpha_k)} : =  E_{\alpha_{k} \alpha_q} f_{\alpha_0 \dots \alpha_k \dots \alpha_{q}}.$$

Let again $\pi: V \to M$ be a rank-$r$ vector bundle over a $m$-dimensional 
nonsingular projective variety $M$ over $\mathbb{C}$.
Let $\mathcal{U}=\{U_\alpha\}$ be a sufficiently fine open covering for $M$ over which $V$ trivializes, 
with local coordinates $\{ z_\alpha = (z_\alpha^1, \dots, z_\alpha^m) \}$. 
Let $\{f_{\alpha \beta}(z_\beta) = (f^1_{\alpha \beta}(z_\beta), \dots, f^m_{\alpha \beta}(z_\beta) )\}$ 
be the transition functions for the given atlas
and $V_{\alpha \beta}(z_\beta)$ the matrix of transition functions for $V$.
An open covering of the total space of $V$ 
is given by $\mathcal{W} = \{W_\alpha \}$, where  $W_\alpha:= U_\alpha \times \mathbb{C}^r$,
and the transition functions on $W_{\alpha}\cap W_{\beta}$ are
\begin{eqnarray}\label{eq:vectorbundlecoordinates}
\left\{\begin{array}{rcl}
   z_{\alpha}  &=& f_{\alpha \beta} (z_\beta)  \\ 
   w_{\alpha} &=& V_{\alpha\beta} (z_\beta) \,  w_{\beta}   
\end{array}\right.
\end{eqnarray}
where $w_\alpha = \left( w^1_\alpha, \dots, w^r_\alpha \right)$ are coordinates
on the fibre $\mathbb{C}^r$.

The first order deformations of the total space of $V$, i.e.  
a fibre $p^{-1}(\epsilon)$ of a family $p:\mathcal{V} \to \mathrm{Spec} (\mathbb{C}[\epsilon] / \epsilon^2)$
with central fibre $V$, can be represented by the transition functions
on $W_{\alpha}\cap W_{\beta}$
\begin{eqnarray}
   \left(\begin{array}{c}
          z_{\alpha} \\
          w_{\alpha}
         \end{array} \right) =    
   \left(\begin{array}{c}
          f_{\alpha \beta} (z_\beta) \\
          V_{\alpha\beta} (z_\beta)  w_{\beta}
         \end{array} \right) +
   \epsilon M_{\alpha\beta} (z_\beta, w_\beta)
   \left(\begin{array}{c}
           \zeta_{\alpha\beta} \\
           \omega_{\alpha \beta}
         \end{array} \right) 
\end{eqnarray}
where $\{ (  \zeta_{\alpha\beta},  \omega_{\alpha\beta} )\}$ are regular
on $W_{\alpha}\cap W_{\beta}$, 
and $\epsilon^2 = 0$.
For convenience, we multiplied the deformation term by
\begin{eqnarray}
M_{\alpha\beta} :=
    \left(\begin{array}{cc}
          J_{\alpha\beta} & 0 \\
          R_{\alpha\beta} & V_{\alpha\beta}
         \end{array} \right)  ,
\end{eqnarray}
where $V_{\alpha \beta}$ is the matrix of transition functions for $V$, 
and $ J_{\alpha \beta}$ and $R_{\alpha \beta}$ are the $m \times m$ and $m \times r$ 
matrices with coefficients
\begin{eqnarray}
\left(J_{\alpha \beta}\right)^a_{\phantom{a}b} : = \frac{\partial f^a_{\alpha \beta}}{\partial z^b_\beta} \ ,  \qquad 
\left(R_{\alpha \beta}\right)^i_{\phantom{i}b} : = 
              \sum_{j=1}^{r} \frac{\partial \left(V_{\alpha \beta}\right)^i_{\phantom{i}j}}{\partial z^b_\beta} w^j_\beta  
\end{eqnarray}
with $a, b = 1, \dots, m$ and  $i, j = 1, \dots, r$. Notice that $J_{\alpha\beta}$ is the matrix
of transition functions of $T_M$ on $U_{\alpha}\cap U_{\beta}$ and $M_{\alpha\beta}$ the one of $T_V$ on 
$W_{\alpha}\cap W_{\beta}$.

As expected, we have the following
\begin{lemma}
  The deformation terms $\{(\zeta_{\alpha\beta}, \omega_{\alpha\beta})\}$
  are cocycles representing elements in the \v{C}ech cohomology group $\check{H}^1(V, T_V)$.
\end{lemma}
\begin{proof}

The compatibility on triple intersections gives
\begin{eqnarray}
  J_{\alpha\gamma}  \zeta_{\alpha\gamma} &=&  J_{\alpha\gamma}  \zeta_{\beta\gamma} 
                                                   + J_{\alpha\beta}  \zeta_{\alpha\beta} \nonumber \\
  R_{\alpha\gamma} \zeta_{\alpha\gamma}  + V_{\alpha\gamma} \omega_{\alpha\gamma} &=& 
                 R_{\alpha\gamma} \zeta_{\beta\gamma} + V_{\alpha\gamma} \omega_{\beta\gamma} + 
                 R_{\alpha\beta} \zeta_{\alpha\beta} + V_{\alpha\beta} \omega_{\alpha\beta} \nonumber
\end{eqnarray}
i.e. the cocycle conditions
\begin{eqnarray}
   \left(\begin{array}{c}
          \zeta_{\alpha\gamma} \\
          \omega_{\alpha\gamma}
         \end{array} \right) =    
   \left(\begin{array}{c}
          \zeta_{\beta\gamma} \\
          \omega_{\beta\gamma}
         \end{array} \right) +
   M_{\gamma\beta}
   \left(\begin{array}{c}
           \zeta_{\alpha\beta} \\
           \omega_{\alpha \beta}
         \end{array} \right) \nonumber
\end{eqnarray}
If we impose the equivalence of two deformations when they are related by a change of local coordinates of the form
\begin{eqnarray}
   \left(\begin{array}{c}
          z_{\alpha}' \\
          w_{\alpha}'
         \end{array} \right) :=    
   \left(\begin{array}{c}
          z_\alpha \\
          w_{\beta}
         \end{array} \right) +
   \epsilon 
   \left(\begin{array}{c}
           h_{\alpha}(z_\alpha,w_\alpha) \\
           g_{\alpha}(z_\alpha,w_\alpha)
         \end{array} \right) \nonumber
\end{eqnarray}
we obtain the coboundary conditions
\begin{eqnarray}
   \left(\begin{array}{c}
          \zeta_{\alpha\beta}' \\
          \omega_{\alpha\beta}'
         \end{array} \right) -   
   \left(\begin{array}{c}
          \zeta_{\alpha\beta} \\
          \omega_{\alpha\beta}
         \end{array} \right) =    
   \left(\begin{array}{c}
          h_\alpha \\
          g_{\beta}
         \end{array} \right) -
   M_{\alpha\beta} 
   \left(\begin{array}{c}
           h_{\beta} \\
           g_{\beta}
         \end{array} \right)  \nonumber
\end{eqnarray}
and the statement.
\end{proof}

Let as before $\pi: V \to M$ be a vector bundle and let us consider another 
vector bundle $E \to M$ and its pullback bundle $\pi^* E \to V$.
A section $\{ \eta_\alpha (z_\alpha, w_\alpha) \}$ of $\pi^* E$ on $\cup W_\alpha \simeq \cup ( U_\alpha \times \mathbb{C}^r)$
can be expanded in polynomials of degree $d$ in $\{ w_\alpha \}$ as
\begin{eqnarray}\label{eq:polyeta}
  \eta_\alpha (z_\alpha, w_\alpha) = \sum_d \varepsilon_\alpha^{(d)} (z_\alpha) \cdot (w_\alpha)^{\otimes d}   
\end{eqnarray}
where $\varepsilon_\alpha^{(d)} (z_\alpha)$ is a (symmetric) tensor-valued function on $U_\alpha$ and
\begin{eqnarray}
  \varepsilon_\alpha^{(d)} \cdot (w_\alpha)^{\otimes d} := 
                              \sum_{i_1, \dots, i_d = 1}^{r} (\varepsilon_\alpha^{(d)})_{i_1 i_2\dots i_d} w_\alpha^{i_1} \cdots w_\alpha^{i_d} .
\end{eqnarray}
On $U_{\alpha\beta}$ one has 
\begin{eqnarray}
  (\varepsilon_\alpha^{(d)})_{i_1 \dots i_d} =  \sum_{j_1, \dots, j_d = 1}^{r}
                  E_{\alpha \beta}
                  (V^*_{\alpha\beta})_{i_1}^{\phantom{i_1}j_1} \cdots  (V^*_{\alpha\beta})_{i_d}^{\phantom{i_d}j_d}
                            (\varepsilon_\beta^{(d)})_{j_1 \dots j_d} \nonumber
\end{eqnarray}
where $E_{\alpha \beta}$ are the transition functions for $E$ and $V^*_{\alpha\beta} = (V^T_{\alpha\beta})^{-1}$, 
i.e. $\{ \varepsilon_\alpha^{(d)} \}$ is a section of $E \otimes \Sym^d V^*$ on $\cup U_\alpha$.
It is easy to see that the previous expansion, when applied to \v{C}ech cocycles, 
commutes with the \v{C}ech differential and gives the explicit isomorphism of 
Lemma \ref{lem:isomorph}.  

In particular, we can apply the previous considerations to sections 
$\{ \zeta_\alpha (z_\alpha, w_\alpha) \}$ and $\{ \omega_\alpha (z_\alpha, w_\alpha) \}$,
respectively of $\pi^* T_\Sigma$ and  of $\pi^* V$ on $\cup W_\alpha \simeq U_\alpha \times \mathbb{C}^r$,
and expand them in polynomials as
\begin{eqnarray}\label{eq:polyexpansion}
  \zeta_\alpha (z_\alpha, w_\alpha) = \sum_d \rho_\alpha^{(d)} (z_\alpha) \cdot (w_\alpha)^{\otimes d}  \nonumber \\
  \omega_\alpha (z_\alpha, w_\alpha) = \sum_d \sigma_\alpha^{(d)} (z_\alpha) \cdot (w_\alpha)^{\otimes d} .
\end{eqnarray}

The Atiyah sequence on the \v{C}ech complexes of \v{C}ech cochains reads
\begin{eqnarray}
 0 \to \check{C}^{\bullet}(\mathcal{W}, \pi^* V) \to \check{C}^{\bullet}(\mathcal{W}, T_V)
   \to \check{C}^{\bullet}(\mathcal{W}, \pi^* T_M) \to 0 \nonumber
\end{eqnarray}
where
\begin{eqnarray}
 i^{\#}_q &:& \left\{ \omega_{\alpha_0 \alpha_1 \dots \alpha_q} \right\} 
              \to \left\{ (0, \omega_{\alpha_0 \alpha_1 \dots \alpha_q})  \right\} \nonumber \\
 p^{\#}_q &:& \left\{ (\zeta_{\alpha_0 \alpha_1 \dots \alpha_q},  \omega_{\alpha_0 \alpha_1 \dots \alpha_q} ) \right\} 
              \to \left\{ \zeta_{\alpha_0 \alpha_1 \dots \alpha_q} \right\}  \nonumber
\end{eqnarray} 
and, as usual, they induce maps in cohomology. The connecting morphisms are
\begin{eqnarray}
  \delta_q : \check{H}^q (\mathcal{W}, \pi^* T_M) \to \check{H}^{q+1} (\mathcal{W}, \pi^* V) \nonumber
\end{eqnarray}
with
\begin{eqnarray}
  \left.\delta_q \zeta\right|_{\alpha_0 \dots \alpha_{q+1}} = (-)^{q+1} R_{\alpha_{q+1}\alpha_q} 
                                                   \zeta_{\alpha_0 \dots \alpha_{q}} . \nonumber
\end{eqnarray}
Finally, let us notice that they can be decomposed, according to (\ref{eq:polyexpansion}), as
$\delta_q = \sum_d \delta_q^{(d)}$, with
\begin{eqnarray}
  \delta_q^{(d)} : \check{H}^q (\mathcal{U}, T_M \otimes \Sym^d V^*) \to
                   \check{H}^{q+1} (\mathcal{U}, V \otimes \Sym^{d+1} V^*),  \nonumber
\end{eqnarray}
where the maps $\delta_q^{(d)}$ act on cocyles as
\begin{eqnarray}
  \left( \left. \delta_q^{(d)} \rho^{(d)} \right|_{\alpha_0 \dots \alpha_{q+1}} \right)^i_{i_1 \dots i_{d+1}}
             = (-)^{q+1} \sum_{b=1}^{n} \frac{\partial (V_{\alpha_{q+1}\alpha_q})^i_{i_1} }{\partial z_{\alpha_q}^b} 
                           (\rho^{(d)}_{\alpha_0 \dots \alpha_{q}})^b_{i_2 \dots i_{d+1}} . \nonumber
\end{eqnarray}

\section{Calabi--Yau condition}\label{sec:CYcondition}

Let as before $\pi: V \to M$ be a rank-$r$ vector bundle over a $m$-dimensional 
nonsingular projective variety $M$ over $\mathbb{C}$.

The Calabi--Yau condition for the total space of $V$, 
i.e. the triviality of its canonical bundle $K_V \simeq \mathcal{O}_V$,
is equivalent to the condition $\det V \simeq K_M$.
This can be seen for example by using the Atiyah sequence of vector bundles on $V$ and
\begin{eqnarray}
  K_V \simeq (\det T_V)^{-1} \simeq \pi^* \left( (\det V)^{-1} K_M \right) .
\end{eqnarray}
In local coordinates the Calabi--Yau condition gives the
condition on the cocycles $\{J_{\alpha\beta}\}$ and $\{V_{\alpha\beta}\}$
\begin{eqnarray}
  \det J_{\alpha\beta} \det V_{\alpha\beta} = 1 \ .
\end{eqnarray}

\subsection{Another map in cohomology}

We now want to consider, among the first-order deformations of the total space of $V$,
those that do not deform the canonical bundle. 
In particular, in the case in which the total space of $V$ has trivial canonical bundle, these deformations
leave the canonical bundle trivial.
We have the following proposition.

\begin{lemma}
  The map
\begin{eqnarray}
  \partial^{\#} : H^1 (V, T_V) \to H^1 (V, \mathcal{O}_V) \ ,
\end{eqnarray}
induced in cohomology by the differential map (\emph{divergence})
\begin{eqnarray}
  \partial: T_V \to \mathcal{O}_V
\end{eqnarray}
associates to a first-order deformation $\delta_\epsilon V$ of the total space of $V$, 
a first-order deformation $\delta_\epsilon K_V$ of the canonical bundle $K_V$ of the total space of $V$.
Its kernel gives the deformations that leave invariant the canonical bundle
(\emph{Calabi--Yau deformations}).

\end{lemma}

\begin{proof}
Recall the transition functions for $V_\epsilon$ (with $\epsilon^2 = 0$)
\begin{eqnarray}
   \left(\begin{array}{c}
          z_{\alpha} \\
          w_{\alpha}
         \end{array} \right) =    
   \left(\begin{array}{c}
          f_{\alpha \beta} (z_\beta) \\
          V_{\alpha\beta} (z_\beta)  w_{\beta}
         \end{array} \right) +
   \epsilon M_{\alpha\beta} (z_\beta, w_\beta)
   \left(\begin{array}{c}
           \zeta_{\alpha\beta} \\
           \omega_{\alpha \beta}
         \end{array} \right) \nonumber
\end{eqnarray}
where $\{(\zeta_{\alpha\beta}, \omega_{\alpha\beta})\}$ are cocycle representatives of elements of $\check{H}^1 (V, T_V)$.
The transition functions for the top forms are
\begin{eqnarray}
 && \de z^1_\alpha \wedge \cdots \wedge \de z^m_\alpha \wedge \de w^1_\alpha \wedge \cdots \wedge \de w_\alpha^r = \nonumber\\
 &&\quad \det M_{\alpha\beta} \det \left( 1 + \epsilon M_{\alpha\beta}^{-1} N_{\alpha \beta} \right) 
  \de z^1_\beta \wedge \cdots \wedge \de z^m_\beta \wedge \de w^1_\beta \wedge \cdots \wedge \de w^r_\beta  , \nonumber
\end{eqnarray}
where
\begin{eqnarray}
 N_{\alpha\beta} =  \left( \begin{array}{ll} 
                            \frac{\partial}{\partial z_\beta} &  \frac{\partial}{\partial w_\beta}
                         \end{array} \right) \otimes M_{\alpha\beta} 
                      \left( \begin{array}{l} 
                             \zeta_{\alpha\beta} \\
                             \omega_{\alpha\beta}
                         \end{array} \right)   . 
\nonumber
\end{eqnarray}
Moreover, since $\epsilon^2 = 0$, we have the identity
\begin{eqnarray}
  \det \left( 1 + \epsilon M^{-1} N \right) = 
         \left( 1 + \epsilon \Tr M^{-1} N \right) .
         \nonumber
\end{eqnarray}
Notice now that we can write the term $\Tr M^{-1} N$, giving the first-order deformations of the canonical bundle, as
\begin{eqnarray}
\Tr M^{-1}_{\alpha\beta} N_{\alpha\beta} = 
                            \sum_{a=1}^{m} \frac{\partial \zeta^{(\alpha) a}_{\alpha\beta}}{\partial z^a_\alpha}
                                  + \sum_{k=1}^{r} \frac{\partial \omega^{(\alpha)k}_{\alpha\beta}}{\partial w^k_\alpha} \nonumber
\end{eqnarray}
where the superscript $(\alpha)$ on $\zeta^{(\alpha) a}_{\alpha\beta}$ and $\omega^{(\alpha)k}_{\alpha\beta}$ means that the cocycles
are represented in the trivialization for the chart $U_\alpha$.
In fact, 
\begin{eqnarray}
  \Tr (M^{-1}_{\alpha\beta} N_{\alpha\beta}) =
                      \left< M_{\alpha\beta}^{-1 T} (\frac{\partial}{\partial z_\beta}, \frac{\partial}{\partial w_\beta}), 
                                 M_{\alpha\beta} (\zeta_{\alpha\beta}, \omega_{\alpha\beta}) \right> . \nonumber
\end{eqnarray}
In the $U_\beta$-trivialization we have
\begin{eqnarray}
 \Tr (M^{-1}_{\alpha\beta} N_{\alpha\beta}) =
            \sum_{a=1}^m \frac{\partial \zeta_{\alpha\beta}^{a}}{\partial z_\beta^a} +  
                        \frac{\partial \kappa_{\alpha\beta}}{\partial z_\beta^a} \zeta_{\alpha\beta}^{a} 
            + \sum_{k=1}^{r} \frac{\partial \omega^{k}_{\alpha\beta}}{\partial w^k_\beta}      
           \nonumber
\end{eqnarray}
where $\kappa_{\alpha\beta}:=\ln (\det J_{\alpha\beta} \det V_{\alpha\beta})$, i.e. $\kappa_{\alpha\beta}=0$ in
the case in which $V$ is Calabi--Yau.
\end{proof}

\subsection{The condition reduced to the base}

Given a vector bundle $V$ and its dual $V^*$, we have natural contraction maps
(also called internal products)
\begin{eqnarray}
  c_{(p,q)}: \Sym^p V \otimes \Sym^{p+q} V^* \to \Sym^{q} V^* . \nonumber
\end{eqnarray}

Going back to the Calabi--Yau condition, we have the following.

\begin{lemma}
 The map 
\begin{eqnarray}
  \left. \partial^{\#} \circ p^{\#} \right|_d : H^1(M, V \otimes \Sym^{d} V^*) \to H^1(M, \Sym^{d-1} V^*)  \nonumber
\end{eqnarray}
is the map induced in cohomology by the contraction map $c_{(1,d-1)}$.
\end{lemma}

\section{Dimension one} 

Let $\pi : V \to \Sigma$ a rank $k$ vector bundle on a nonsingular curve $\Sigma$.
In this case, the exact sequence of Lemma \ref{exactseq} stops in grade $1$
and we do not have higher obstructions. Now, we could repeat step-by-step the 
passages of the previous sections, considering the simplifications of the case 
of a curve. Instead, we take a slightly different viewpoint.

Let us first remark that to study deformations of the total space of $V$, we can reduce 
to the case in which $V$ is a sum of line bundles, in the following sense.
Since every vector bundle of rank $r>1$ over a nonsingular curve contains
a line bundle as a sub-bundle,
we have a sequence of inclusions of vector sub-bundles
\begin{eqnarray}
\phi_1 = V_1 \subset V_2 \subset \dots \subset V_{r-1} \subset V_r = V \nonumber
\end{eqnarray}
and extensions 
\begin{eqnarray}
  && 0 \to \phi_1 \to V_2 \to \phi_2 \to 0 \nonumber \\
  && 0 \to V_2 \to V_3 \to \phi_3 \to 0 \nonumber \\
  && \cdots \nonumber \\
  && 0 \to V_{r-1} \to V \to \phi_r \to 0 \ . \nonumber
\end{eqnarray}
This implies, by using the openness of versality theorem, 
see \cite{Artin}, a sequence of inclusions of versal deformations
\begin{eqnarray}\label{incl}
  \Def (V) \subset \Def (V_{r-1} \oplus \phi_r) \subset \cdots \subset
                    \Def (\phi_1 \oplus \cdots \oplus \phi_r) \ . \nonumber
\end{eqnarray}

In the case in which $V \simeq V_0 = \phi_1 \oplus \dots \oplus \phi_r$, the transition functions for the deformation have the form
\begin{eqnarray}
  \left\{ 
     \begin{array}{rcl}
        z_\alpha    &=& f_{\alpha \beta} (z_\beta) + \epsilon f'_{\alpha \beta} (z_\beta) \zeta_{\alpha\beta}(z_\beta, \omega_\beta) \\
        w^j_\alpha  &=& \phi^j_{\alpha\beta} (z_\beta) \left( w^j_\beta + \epsilon \omega^j_{\alpha\beta}(z_\beta, w_\beta) + 
                                        \epsilon (\phi^i_{\alpha\beta})^{-1}\frac{\partial \phi^i_{\alpha\beta}}{\partial z_\beta} w^i_\beta \zeta_{\alpha\beta}\right) 
\\
     \end{array}
  \right. \nonumber \\
 \qquad 
                            \mathrm{for} \ j=1, \dots, r , \nonumber
\end{eqnarray}
where $\{\phi^j_{\alpha\beta}\}$ are the transition functions for the line bundles $\phi_j$. 
Moreover we have
\begin{eqnarray}\label{eq:symsplit}
  \Sym^{d} \left( \phi_1 \oplus \cdots \oplus \phi_r \right)^{*} = 
               \bigoplus_{\genfrac{}{}{0pt}{}{m_1, \dots, m_r=0}{m_1 + \cdots + m_r = d}}^{d} \phi_1^{-m_1} \cdots \phi_r^{-m_r} .  
\end{eqnarray}

If we write the deformations in the form
\begin{eqnarray}
   \zeta_{\alpha\beta} (z_\beta, w_\beta) &: =& \sum_{m_1,\dots, m_r}  \rho^{(m_1 \cdots m_r)}_{\alpha\beta} (z_\beta)     (w^1_\beta)^{m_1} \cdots (w^r_\beta)^{m_r} \nonumber \\
   \omega^j_{\alpha\beta} (z_\beta, w_\beta) &: =& \sum_{m_1,\dots, m_r}  \sigma^{j (m_1 \cdots m_r)}_{\alpha\beta}  (z_\beta)   (w^1_\beta)^{m_1} \cdots (w^r_\beta)^{m_r} \nonumber 
\end{eqnarray}
then we have, as a simple corollary of the previous lemmas and by using (\ref{eq:symsplit}), or by direct computation, 
\begin{eqnarray}
\{ \rho^{(m_1 \cdots m_r)}_{\alpha\beta} \} \in  H^1(\Sigma, K_\Sigma^{-1} \phi_1^{-m_1} \cdots \phi_r^{-m_r}) \nonumber 
\end{eqnarray}
and, for $\{\rho_{\alpha\beta}^{(m_1 \cdots m_j-1 \cdots m_r)}\}=0$,
\begin{eqnarray}
  \{ \sigma^{j (m_1 \cdots m_r)}_{\alpha\beta}\} \in  H^1(\Sigma,  \phi_1^{-m_1}\cdots \phi_j^{-(m_j-1)}\cdots \phi_r^{-m_r} ) .  \nonumber
\end{eqnarray}

The Calabi--Yau condition on each intersection $U_{\alpha \beta}$ reads
\begin{eqnarray}
  \frac{\partial \zeta_{\alpha\beta}}{\partial z_\beta} 
                  + \left( \frac{\partial}{\partial z_\beta} \ln f'_{\alpha\beta} \phi^1_{\alpha\beta}\cdots\phi^r_{\alpha\beta}\right) \zeta_{\alpha\beta} 
                  + \sum_{k=1}^{r}  \frac{\partial \omega_{\alpha\beta}^k}{\partial w_\beta^k} = 0 ,
\end{eqnarray}
that is, for the Calabi--Yau case $f' \phi^1\cdots\phi^r = 1$, 
\begin{eqnarray}
  \frac{\partial \zeta_{\alpha\beta}}{\partial z_\beta} + \sum_{k=1}^{r}  \frac{\partial \omega_{\alpha\beta}^k}{\partial w_\beta^k} = 0 .
\end{eqnarray}
Expanding in powers of $w$ as before, the Calabi--Yau condition becomes
\begin{eqnarray}
  \frac{\partial}{\partial z_\beta}  \rho^{(m_1 \cdots m_r)}_{\alpha\beta}
     = - \sum_{j=1}^r (m_j+1) \sigma^{j (m_1 \cdots m_j + 1 \cdots m_r)}_{\alpha\beta} . 
\end{eqnarray}

\paragraph{Acknowledgements.}
The author is grateful to
Giulio Bonelli, Ugo Bruzzo and Barbara Fantechi, for helpful discussions, 
to Marco Pacini, for participating to an earlier stage of this project,
and especially to Richard Thomas, for illuminating comments on the manuscript.
The author is supported in part by a fellowship from 
Fondazione Angelo della Riccia, Firenze.

\end{document}